\documentclass[10pt]{amsart}
\usepackage{amsmath,amssymb,xypic,array}
\usepackage[T1]{fontenc}
\usepackage{amsfonts,amsthm,epsfig}
\usepackage{setspace}
\usepackage{graphicx}
\usepackage{tikz}
\usepackage{booktabs}
\usepackage{longtable}
\usepackage{mathtools}
\usetikzlibrary{trees}
\usepackage{geometry}
\usepackage{tikz-cd}
\usepackage[english]{babel}
\usepackage[latin1]{inputenc}
\usepackage{color}
\usepackage{mathrsfs}
\usepackage{latexsym}
\usepackage{hyperref}
\usepackage{textcomp}

\newtheorem{theorem}{Theorem}
\newtheorem{lemma}[theorem]{Lemma}
\newtheorem{proposition}[theorem]{Proposition}
\newtheorem{definition}[theorem]{Definition}
\newtheorem{corollary}[theorem]{Corollary}
\newtheorem{remark}[theorem]{Remark}

\DeclareMathOperator{\rk}{{rk}}

\DeclareMathOperator{\coker}{{coker}}

\DeclareMathOperator{\Hom}{Hom}
\DeclareMathOperator{\Ext}{Ext}

\title[On rank 3 instanton bundles on $\mathbb{P}^3$]
{On rank 3 instanton bundles on $\mathbb{P}^3$}

\author[A. V. Andrade]{A. V. Andrade}
\thanks{ }
\dedicatory{}
\address{Instituto de Ci\^encias Exatas\\ Departamento de Matem\'atica\\ Universidade Federal de Minas Gerais\\ Belo Horizonte, MG, 30161-970, Brazil}
\email{andradealine@mat.ufmg.br}

\author[ D. R. Santiago  ]{ D. R. Santiago }
\thanks{ }
\dedicatory{}
\address{Instituto Federal de Sergipe - Campus Gl\'oria \\ Coordena\c c\~ao de Agropecu\'aria \\
Povoado Piabas, Zona Rural S/N \\ 49680-000 Nossa Senhora da Gl\'oria-SE, Brazil}
\email{danilo.santiago@ifs.edu.br}

\author[ D. D. Silva  ]{ D. D. Silva }
\thanks{ }
\dedicatory{}
\address{Departamento de Matem\'atica - UFS \\ Avenida Marechal Rondon S/N \\ S\~ao Cristov\~ao-SE, Brazil}
\email{ddsilva@ufs.br}

\author[  L. C. S. Sobral]{L. C. S. Sobral}
\thanks{ }
\dedicatory{}
\address{Instituto de Matem\'atica, Estat\'istica e Computa\c c\~ao Cient\'ifica - UNICAMP \\ Departamento de Matem\'atica \\
Rua S\'ergio  Buarque de Holanda, 651\\ 13083-970 Campinas-SP, Brazil}
\email{luisufspaiandre@gmail.com}

\date{}

\begin{document}

\begin{abstract}
We investigate rank $3$ instanton vector bundles on $\mathbb{P}^3$ of charge $n$ and its correspondence with rational curves of degree $n+3$. For $n=2$ we present a correspondence between stable rank $3$ instanton bundles and stable rank $2$ reflexive linear sheaves of Chern classes $(c_1,c_2,c_3)=(-1,3,3)$ and we use this correspondence to compute the dimension of the family of stable rank $3$ instanton bundles of charge $2$. Finally, we use the results above to prove that the moduli space of rank $3$ instanton bundles on $\mathbb{P}^3$ of charge $2$ coincides with the moduli space of rank $3$ stable locally free sheaves on $\mathbb{P}^3$ of Chern classes $(c_1,c_2,c_3)=(0,2,0)$.  This moduli space is irreducible, has dimension 16 and its generic point corresponds to a \textcolor{black}{generalized} t`Hooft instanton bundle.
\medskip
		
		
		\medskip
		
\end{abstract}

\maketitle

\section{Introduction}

Instanton bundles have attracted a great deal of attention from the mathematical community since the 1970's not only because of the connection that they provide with physics (see \cite{AHS1977} and \cite{AW1977}) as well for their applications to the study of moduli spaces of stable vector bundles on projective varieties. Although the family of  instanton bundles \textcolor{black}{has been know} for quite some time and besides the fact that there is a vast literature on the subject (see \cite{ADHM1978}, \cite{CTT2003}, \cite{J2006}, \cite{OSS1980}, \cite{P1977} to mention a few), it was only in the 2010's that it was proved that \textcolor{black}{the moduli space of stable rank 2 vector instanton bundles of charge $n$ is a smooth (see \cite{JV2014}) and irreducible (see \cite{T2012} and \cite{T2013}) open subset of the moduli space of stable rank 2 bundles with Chern classes $c_1=0$ and $c_2=n$ on $\mathbb{P}^3$}. Since then, several authors pursued what they considered to be \textcolor{black}{a natural direction for the study of instantons}: some studied their definition for other projective varieties ( see \cite{AM2022}, \cite{CG2020}, \cite{F2014} \cite{MMPL2020}); some considered their \textcolor{black}{degenerations} in $\mathbb{P}^3$ (see \cite{GJ2016}, \cite{JMT2018}); and others considered the higher rank case (\cite{AMMR2020}, \cite{BMT2012}, \cite{JMW2016}).

In this work, we will focus on the study of rank $3$ instanton bundles on $\mathbb{P}^3$ hoping that this family of bundles \textcolor{black}{provides} a good candidate - as it was for the rank $2$ case (see \cite{CG2020}, \cite{T2012}, \cite{T2013}) - for the study of irreducible components of the moduli space of stable rank $3$ bundles on $\mathbb{P}^3$. 

We start this investigation exploiting a connection between $\mathcal{I}(n)$, the family of rank $3$ instanton bundles of charge $n \geq 2 $, and rational curves of degree $n+3$. We use this correspondence, in Lemma \ref{lem01},  to concretely obtain, for $n \geq 2$, a family of rank $3$ instanton bundles $F$ for which we have $h^0(F(1)) \geq 2$ and $\Ext^2(F,F)=0$ generically. In other words, the stable rank $3$ generalized 't Hooft bundle of charge $n$ (see Definition \ref{hooft}) is a smooth point in $\mathcal{B}(0,n,0)$, the moduli space of stable rank $3$ vector bundles with Chern classes $(c_1,c_2,c_3)=(0,n,0)$.

Specializing to the case $n=2$, we prove (Theorem \ref{correspondence1}) that there exists \textcolor{black}{a} one-to-one correspondence between:
\begin{itemize}
\item[1.] Pairs $(F, \sigma)$, where $F$ is a rank $3$ stable instanton bundle of charge $2$ on $\mathbb{P}^3$ and $\sigma$ is a global section of $F(1)$. 

\item[2.] Pairs $(E,\xi)$ where $E$ is normalized rank $2$ stable reflexive sheaf on $\mathbb{P}^3$ with Chern classes $(-1,3,3)$ and $\xi \in \Ext^{1}(E(2),\mathcal{O}_{\mathbb{P}^{3}})$.
\end{itemize} 

This correspondence is crucial to find the dimension of $\mathcal{I}(2)$ using an earlier result regarding moduli spaces of rank $2$ reflexive sheaves on $\mathbb {P}^3$  by Chang \cite{C1984}. Moreover, we prove that $\mathcal{I}(2)$ and $\mathcal{B}(0,2,0)$ coincides. Actually, we have our main result:\\

\textcolor{black}{{\bf Main Theorem.}}
$\mathcal{B}(0,2,0)$ is irreducible of dimension $16$ and its generic point corresponds to a \textcolor{black}{generalized} t`Hooft instanton bundle.\\

This paper is organized as follows. In Section \ref{preliminaries} we revise some definitions and key facts about the Hartshorne-Serre correspondence, instanton sheaves on $\mathbb{P}^3$ and the concept of spectrum of a torsion free sheaf. Also in \textcolor{black}{Section \ref{preliminaries}}, we construct explicitly rank $3$ instanton bundles of charge $n \geq 2 $ from rational curves of degree $n+3$. Following the ideas of Hartshorne in \cite{H1978}, we then exploit, in Section \ref{invariants}, the connection between rank $3$ bundles and curves to relate the numerical invariants of the bundle to the invariants of the curve. In Section \ref{correspondence} we prove that the generic stable rank $3$ generalized 't Hooft bundle of charge $n$ is a smooth point in the moduli space of rank $3$ vector bundles; we also present a correspondence between stable rank $3$ instanton bundles on $\mathbb{P}^3$ of charge $2$ and rank $2$ stable reflexive linear sheaves on $\mathbb{P}^3$ of Chern classes $(-1,3,3)$. Finally, in Section \ref{moduli}, we apply the results found in Sections \ref{invariants} and \ref{correspondence} to prove that $\mathcal{B}(0,2,0)$ is irreducible of dimension 16 and its generic point is a \textcolor{black}{generalized} t`Hooft instanton bundle.

\subsection*{Notation and conventions} In  this work, we will not make any distinction between vector bundles and locally free sheaves. If $F$ is a sheaf on $\mathbb{P}^n$, then $h^i(F)$ denotes the dimension of $H^i(F)$, \textcolor{black}{the $i^{th}$ cohomology group}, and $F^{*}$ denotes the dual of $F$, i. e., $F^{\ast} =\mathcal{H}om(F,\mathcal{O}_{\mathbb{P}^n})$. \textcolor{black}{Unless stated otherwise a sheaf is said to be (semi)stable if it is Gieseker (semi)stable.}

\subsection*{Acknowledgments}
The authors would like to thank both referees for the thorough reading of our work and most of all for the suggestions that helped us to improve it. We also would like to express our gratitude to Marcos Jardim for his invaluable support during the realization of this project.
 The first author was partially supported by CAPES-COFECUB project number 88881.191920/2018-01 and FAPEMIG Universal Grant APQ-01655-22. The second named author was supported by the CNPQ grant number 141174/2019-6. The fourth named author was supported by the CNPQ grant number 140727/2016-7 during his Phd whose thesis inspired part of this paper.

\section{Hartshorne-Serre correspondence and instanton bundles}\label{preliminaries} Roughly speaking, the Hartshorne-Serre correspondence consists on associating subschemes $Y$ of codimension $2$ in a nonsingular algebraic variety $X$ to pairs $(F,\sigma)$, where $F$ is a vector bundle and $\sigma$ is a section of $F$. In \cite{H1980b}  Hartshorne proved that this correspondence works in case $X = {\mathbb{P}}^{3}$, $Y$ is a locally complete intersection curve in $X$ and $F$ is a rank $2$ vector bundle, while the general rank case, due to Arrondo, \textcolor{black}{can be found in \cite{A2007}}. One of the features of the Hartshorne-Serre correspondence is the possibility to express a rank $r$ vector bundle $F$ on $X$ by the following exact sequence 

$$0\rightarrow \mathcal{O}^{\oplus(r-1)}_{X}
 \rightarrow F \rightarrow I_{Y}(k)\rightarrow 0,$$

\noindent where $I_{Y}$ is the ideal sheaf of a codimension $2$ subscheme $Y$ in $X$ and $k$ is the first Chern class of $F$. So we see that it is possible to get some information from $F$ in terms of $Y$. In particular, it can be used in the case in which $X = {\mathbb{P}}^{3}$, $F$ is a rank $3$ vector bundle in $X$ and $Y$ is a locally complete intersection curve in $X$. 

\begin{theorem}\label{serre} Let $X$ be a smooth algebraic variety and $Y$ be a codimension $2$ locally complete intersection subscheme. Let $N_{Y/X}$ be the normal bundle of $Y$ in $X$ and $L$ be a line bundle in $X$ such that $H^{2}(X,L^{*})=0$. We assume that $\bigwedge^{2}N_{Y/X} \otimes {L}^{*}|_{Y}$ is generated by $r-1$ global sections $\sigma_{1}, \cdots, \sigma_{r-1}$. Then there exists a rank $r$ vector bundle $F$ on $X$ such that:
\begin{itemize}
\item[i.] $\bigwedge^{r}F = L$

\item[ii.] $F$ has $r-1$ global sections $t_1, \cdots, t_{r-1}$ whose dependency locus is $Y$ and:

$$\sigma_{1}t_{1}+ \cdots +\sigma_{r-1}t_{r-1}=0.$$
\end{itemize}
\noindent If $H^{1}(X,L^{*}) = 0$ then $F$ is unique up to an isomorphism.
\end{theorem}
\begin{proof} \textcolor{black}{See \cite[Theorem 1.1]{A2007}. }\end{proof}

In particular, if $Y$ is a locally complete intersection curve in $\mathbb{P}^3$ such that ${\rm det}N_{Y/\mathbb{P}^3}\otimes{\mathcal{O}_Y(-k)}$ is generated by two global sections then the corresponding rank $3$ vector bundle $F$ has two sections $s_1,s_2$ which define the following short exact sequence

$$\xymatrix{0 \ar[r]& \mathcal{O}_{\mathbb{P}^{3}}^{\oplus 2} \ar[r]^-{(s_1,s_2)}& F \ar[r] &I_{Y}(k)\ar[r]& 0},$$

\noindent so $Y$ is precisely the degeneration locus of $s_1,s_2$. Throughout the rest of this paper we say $Y \subset \mathbb{P}^3$ satisfies the \textcolor{black}{Hartshorne-Serre Correspondence} when $Y$ satisfies the hypotheses of the Hartshorne-Serre Correspondence given by Theorem \ref{serre} for $r=3$.

\textcolor{black}{In the context of Theorem \ref{serre} we say that $F$ is the corresponding bundle of $Y$ via Hartshorne-Serre correspondence.} One of the main features of the \textcolor{black}{Hartshorne-Serre Correspondence} it is that one can provide explicit examples of rank $r$ vector bundles on $\mathbb{P}^3$, which is not a simple task depending on the family of bundles that we are considering (see for instance \cite{AMMR2020} and \cite{BMT2012}). \textcolor{black}{Recall that for the rank $2$ case, the Theorem \ref{serre} can also be extended for reflexive sheaves, as can be seen in \cite[Theorem 4.1]{H1980b}}.  In this work, we will be \textcolor{black}{interested in rank $3$ instanton} bundles on $\mathbb{P}^3$ and to ensure their existence we relate them with rational curves. First, recall the definition of instanton sheaves introduced by Jardim in \cite{J2006}.

\begin{definition} An {\bf instanton sheaf} $F$ on $\mathbb{P}^{3}$ is a torsion-free coherent sheaf satisfying the following cohomology vanishing conditions
$$H^{0}(F(-1))= H^{1}(F(-2)) = H^{2}(F(-2))= H^{3}(F(-3))= 0$$ 

\noindent and such that its first Chern class is 0. Its second Chern class is called the {\bf charge} of $F$. 

\noindent If $F$ is locally free, we call it an {\bf instanton bundle}.
\end{definition}

\begin{remark}
More generally, a coherent sheaf on $\mathbb{P}^3$ which satisfies the cohomology vanishing conditions of the definition is called a {\bf linear sheaf} (see \cite[Proposition 2 and Theorem 3]{J2006}). 

\end{remark}

\begin{theorem}\label{ratinst} Consider a nondegenerate locally complete intersection \textcolor{black}{smooth} rational curve $Y \subset \mathbb{P}^{3}$ of degree $n+3$, $n \geq 2$. Then ${\rm det}N_{Y/\mathbb{P}^{3}} \otimes \mathcal{O}_Y(-3)$ is generated by two global sections and \textcolor{black}{$I_Y(2)$ is a linear sheaf. Furthermore,}  if $F(1)$ is the corresponding rank $3$ vector bundle then $F$ is an instanton bundle of charge $n$.
\end{theorem}

\begin{proof}

From the proof of \cite[Proposition 6.5]{JMT2018}, we have

$$N_{Y/\mathbb{P}^{3}} = \mathcal{O}_{Y}((2n+5)pt)^{\oplus 2},$$

\noindent so we get $\bigwedge^{2} N_{Y/\mathbb{P}^{3}} = O_{Y}((4n+10)pt)$ and hence $ \det N_{Y/\mathbb{P}^{3}} \otimes \mathcal{O}_Y(-3)= \mathcal{O}_Y((n+1)pt)$ which is generated by two global sections. By Hartshorne-Serre correspondence, we have

\begin{equation}\label{auxseq}
0\rightarrow \mathcal{O}_{\mathbb{P}^{3}}^{\oplus 2} \rightarrow F(1) \rightarrow I_{Y}(3) \rightarrow 0
\end{equation}
\noindent where $F(1)$ is the corresponding rank $3$ bundle.

Now we prove that $I_{Y}(2)$ is a linear sheaf. We have $h^{0}(I_{Y}(1)) = 0$ since $Y$ is not contained in a plane. Using the long cohomology sequences of the exact sequence $0\rightarrow I_{Y}  \rightarrow \mathcal{O}_{\mathbb{P}^{3}} \rightarrow \mathcal{O}_{Y} \rightarrow 0$ and of this sequence twisted by $\mathcal{O}_{\mathbb{P}^{3}}(-1)$, we have $h^{2}({I}_{Y})= h^{1}(\mathcal{O}_{Y}) = g = 0$, $h^{3}({I}_{Y})= h^{2}(\mathcal{O}_{Y}) =0$ and $h^{3}(I_{Y}(-1)) = h^{2}(\mathcal{O}_{Y}(-1)) = 0$.

\noindent Finally
$$\chi(I_{Y}) = \chi(\mathcal{O}_{\mathbb{P}^{3}})-\chi(\mathcal{O}_{Y}) = 1 -(0(n+3)+1-0) = 0$$
\noindent implies $h^1(I_Y)=0$ since $h^0(I_Y)= h^2(I_Y)=h^3(I_Y)=0$.

Hence the sheaf $I_{Y}(2)$ is a linear sheaf. Then one uses the exact sequence \eqref{auxseq} to prove that 
$F$ is an instanton sheaf. Indeed, the long cohomology sequences of this sequence twisted by $-2$, $-3$ give us:

\begin{equation}\label{cohoY}
H^0(F(-1))=H^0(I_Y(1))=0, ~
H^1(F(-2))=H^1(I_Y)=0,~
H^2(F(-2))=H^2(I_Y)=0.
\end{equation}

Finallly suppose $H^3(F(-3)) \neq 0$. As $H^3(F(-3))=\Ext^3(\mathcal{O}_{\mathbb{P}^{3}},F(-3))=\Hom(F(1),\mathcal{O}_{\mathbb{P}^{3}})^*$,  there exists a nontrivial map $\sigma: F(1) \rightarrow \mathcal{O}_{\mathbb{P}^{3}}$.

Suppose $\sigma$ is not surjective and consider that the image of $\sigma$ is the ideal sheaf \textcolor{black}{$I_Z(a)$, for some $a\leq 0$.}  Then we have the following diagram

$$\xymatrix{
& & 0 \ar[d] & & \\
0 \ar[r] & \mathcal{O}_{\mathbb{P}^{3}}^{\oplus 2} \ar[r] \ar@{=}[d] & K \ar[r] \ar[d] & E \ar[r] \ar[d] & 0\\
0 \ar[r] & \mathcal{O}_{\mathbb{P}^{3}}^{\oplus 2} \ar[r]^{{\alpha}} \ar[dr]^{0} & F(1) \ar[r] \ar[d]^{\sigma} & I_Y(3) \ar[r]  & 0\\
& & I_Z(a) \ar[d] & & \\
& & 0 & & 
}$$

\noindent where $\sigma\alpha=0$, since \textcolor{black}{$H^0(I_Z(a))=0$.} As $F$ has rank 3 and \textcolor{black}{$I_Z(a)$} rank 1 follows that $K$ has rank 2 and hence $E$ has rank 0. Moreover, by Snake
lemma, $E$ is a subsheaf of $I_Y(3)$, so $E$ is a torsion free sheaf of rank 0 which is a contradiction. 

Now suppose $\sigma$ is surjective. Then we get the following commutative diagram:

$$\xymatrix{
& 0 \ar[d] & 0 \ar[d] & & \\
0 \ar[r] & \mathcal{O}_{\mathbb{P}^{3}} \ar[r] \ar[d] & K \ar[r] \ar[d] & I_Y(3) \ar[r]\ar@{=}[d] & 0\\
0 \ar[r] & \mathcal{O}_{\mathbb{P}^{3}}^{\oplus 2} \ar[r] \ar[d] & F(1) \ar[r] \ar[d] & I_Y(3) \ar[r] & 0\\
 & \mathcal{O}_{\mathbb{P}^{3}} \ar[d] \ar@{=}[r] &  \mathcal{O}_{\mathbb{P}^{3}} \ar[d] & & \\
 & 0  & 0 & & \\
}$$

\noindent where $K$ is locally free. By dualizing the first exact row we get the long exact sequence of cohomology 

$$0 \rightarrow \mathcal{O}_{\mathbb{P}^{3}}(-3) \rightarrow {\mathcal H}{\it om}(K, \mathcal{O}_{\mathbb{P}^{3}}) \rightarrow  \mathcal{O}_{\mathbb{P}^{3}} \rightarrow \mathcal{E}{\it xt}^1(I_Y(3),  \mathcal{O}_{\mathbb{P}^{3}}) \rightarrow 0.$$

Since $\mathcal{E}{\it xt}^{1}(I_{Y}(k),\mathcal{O})\cong  \omega_{Y}(4-k)$, we have $\omega_Y(1)= \mathcal{E}{\it xt}^1(I_Y(3),  {\mathcal O}_{\mathbb{P}^{3}})$ and so we get an epimorphism $\mathcal{O}_{\mathbb{P}^{3}} \rightarrow \omega_Y(1)$. As ${\bigwedge}^{2} N_{Y/\mathbb{P}^{3}} = O_{Y}((4n+10)pt)$, we also get

$$\displaystyle \omega_{Y}(1) =  \omega_{\mathbb{P}^{3}} \otimes {\bigwedge}^{2} N_{Y/\mathbb{P}^{3}}(1) = \mathcal{O}_{Y}((n+1)pt).$$

Hence every section in $H^0(\omega_Y(1))$ vanishes at $n+1$ points and there can be no epimorphism ${\mathcal O}_{\mathbb{P}^{3}} \rightarrow \omega_Y(1)$ which is a contradiction.

\end{proof}

Next we present the definition of spectrum of a torsion free sheaf, introduced by Okonek and Spindler in \cite{OS1984}. This is a very important tool for the classification of bundles, since they provide an effective method to study all possible families of sheaves in their moduli spaces. \textcolor{black}{Before let $E$ be a rank $r$ vector bundle over $\mathbb{P}^n$ and $G_n$ be the Grassmann manifold of lines in $\mathbb{P}^n$. Denote by $\mathit{l}$ the point of $G_n$ which corresponds to a projective line $L\subset \mathbb{P}^n$. By Grothendieck's theorem for every $\mathit{l}\in G_n$ there is a $r$-tuple $a_E(\mathit{l})=(a_1,\cdots,a_r)\in \mathbb{Z}^r$, with $a_1\geq\cdots \geq a_r$ such that $ E|_L \cong \displaystyle \bigoplus_{i=1}^{r}\mathcal{O}_L(a_i)$. The $r$-tuple $a_E(\mathit{l})=(a_1,\cdots,a_r)\in \mathbb{Z}^r$ is called the {\bf splitting type} of $E$ on $L$. By \cite[Lemma 3.2.2]{OSS1980} that is as open dense subset $U\subset G_n$ such that the splitting type of $E$ is constant equals to $\displaystyle a_E=\displaystyle \inf_{\mathit{l}\in G_n} a_E(\mathit{l})$, this is called the {\bf generic splitting type} of $E$.} 

\begin{theorem}\label{spectro}
 Let $F$ be a rank $r$ torsion free sheaf on $\mathbb{P}^3$, with generic splitting type $(a_1,\cdots,a_r)$, with $a_i \in \mathbb{Z}$, and $a_1 \leq a_2 \leq \cdots \leq a_r$, and $s = h^0({\mathcal E}{\it xt}^2(F,\mathcal{O}_{\mathbb{P}^3}))$.
 Then there exists a list of $n$ integers $(k_1, k_2, \cdots ,k_n)$, with $k_1 \leq k_2 \leq  \cdots \leq k_n $ such that
 \begin{itemize}
  \item[a.] $h^1(F(l)) = s + \displaystyle \sum_{i = 1}^{n}h^0(\mathcal{O}_{\mathbb{P}^1}(k_i + l + 1)$ if $l \leq -a_s-1$;
  \item[b.] $h^2(F(l))= \displaystyle \sum_{i = 1}^{n}h^1(\mathcal{O}_{\mathbb{P}^1}(k_i + l + 1)$ if $l \geq a_1 - 3$.
 \end{itemize}
\end{theorem}
\begin{proof}
See \cite[Theorem 2.3]{OS1984}.
\end{proof}

\begin{definition}
Let $F$ be a rank $r$ torsion free sheaf on $\mathbb{P}^3$, with generic splitting type $(a_1,\cdots,a_r)$, with $a_i \in \mathbb{Z}$, and $a_1 \leq a_2 \leq \cdots \leq a_r$, and $s = h^0({\mathcal E}{\it xt}^2(F,\mathcal{O}_{\mathbb{P}^3}))$.
 Then the list of $m$ integers $(k_1, k_2, \cdots ,k_n)$, provided by the previous theorem is called {\bf spectrum} of the sheaf $F$.
\end{definition}

It is possible to prove that the integer $n$ in the definition above is the second Chern class of the sheaf $F$ (see for instance \cite[Proposition 6]{A2020}). For a more detailed study on the spectrum of sheaves see \cite{A2020} or \cite{OS1984}.

Next we present a characterization of instanton bundles by its spectrum.

\begin{proposition}\label{instanton1} Let $F$ be a semistable rank $3$ vector bundle on $\mathbb{P}^{3}$, with $c_1(F)=0$ and $c_2(F)=n$. Then $F$ is an instanton bundle if and only if $k_{F} = \underbrace{(0,\cdots ,0)}_{n} $.
\end{proposition}
\begin{proof} Let $F$ be a rank $3$ semistable  vector bundle of charge $n$ on $\mathbb{P}^{3}$ and let $k_{F}= (k_1,\cdots ,k_{n})$ be its \textcolor{black}{spectrum}. If $F$ is an instanton bundle of charge $n$ on $\mathbb{P}^{3}$ then  $h^{1}(F(-2)) = 0$ which implies, by Theorem \ref{spectro},  

$$ \displaystyle h^{1}(F(-2)) = h^{0}\left(\bigoplus_{i=1}^{n}{\mathcal{O}}_{\mathbb{P}^{1}}(k_{i}-1)\right) = \sum_{i=1}^{n}h^{0}({\mathcal{O}}_{\mathbb{P}^{1}}(k_{i}-1)).$$

 Hence $k_{i} \leq 0$, $ \quad \forall i \in\{1, \cdots, n\}$. By \cite[Theorem 7.3]{OS1984}, $k_{1}+\cdots +k_{n}=0$. As for any \textcolor{black}{spectrum} we have $k_{1} \leq k_{2} \leq \cdots \leq k_{n}$, then  $k_{F} = (0,\cdots ,0)$. 

On the other hand, if $F$ is a rank $3$ semistable vector bundle on $\mathbb{P}^3$, with $c_1(F)=0$, then by \cite[Remark 1.2.6]{OSS1980}, $h^{0}(F(-1)) = h^{0}(F^{*}(-1))=0$. By Serre duality, we get $h^{3}(F(-3))$=0.
Moreover, since $k_{F} = (0,\cdots ,0)$ we get
\begin{itemize}
    \item[i.] $h^{1}(F(-2)) = h^{0}(\oplus_{i=1}^{n}{\mathcal{O}}_{\mathbb{P}^{1}}(-1)) = 0;$
    \item[ii.] $h^{2}(F(-2)) = h^1(E^*(-2))= h^{0}(\oplus_{i=1}^{n}{\mathcal{O}}_{\mathbb{P}^{1}}(-1))=0.$

\end{itemize}
   
Therefore $F$ is an instanton bundle. \end{proof}

\section{Numerical invariants for rank 3 vector bundles} \label{invariants}

Our main goal in this section is to develop tools in order to compute the dimension of the tangent space at a point $F$ in the moduli space of rank $3$ stable vector bundles in terms of the Chern classes of $F$. First, we obtain a relation between the main invariants of a rank $3$ vector bundle on $\mathbb{P}^3$ and its associated curve $Y \subseteq \mathbb{P}^{3}$.

This following result is due to Vogelaar (see \cite{V1978}) but here we present a different proof.

\begin{proposition}\label{chernclass}Let $F$ be a rank $3$ vector bundle on $\mathbb{P}^3$, with Chern classes $c_1,c_2,c_3$, and let $Y \subseteq \mathbb{P}^{3}$ be its corresponding curve, such that $Y$ has degree $d$ and arithmetic genus $g$. Then $d = c_{2}$ and $c_{3}-4c_{2}+ c_{1}c_{2} = 2g-2$, where $c_1,c_2,c_3$ are the Chern classes of $F$.
\end{proposition}
\begin{proof}

Twisting the sequence

$$0 \rightarrow \mathcal{O}_{\mathbb{P}^3}^{\oplus 2} \rightarrow F \rightarrow I_{Y}(c_{1})\rightarrow 0,
$$

\noindent by $\mathcal{O}_{\mathbb{P}^3}(m)$, with $m\in \mathbb{Z}$, we get that 
$$\chi(F(m)) =  2\chi(\mathcal{O}_{\mathbb{P}^3}(m))+\chi(I_Y(c_1+m)).$$

By the exact sequence $0 \rightarrow I_{Y} \rightarrow {\mathcal{O}_{\mathbb{P}^3}} \rightarrow {\mathcal{O}}_{Y} \rightarrow 0$, we obtain \textcolor{black}{$\chi_{\mathcal{O}_{\mathbb{P}^3}}(m)=\chi_{{\mathcal{O}}_{Y}}(m)+\chi_{I_{Y}}(m)$.} Then 

\begin{equation}\label{chi1}
\chi(F(m)) = 2\binom {m+3} {3} + \binom {m+c_{1}+3} {3} - (m+c_{1})d -1+g.
\end{equation}

By Hirzenbruch-Riemann-Roch formula, we also have:
\begin{equation}\label{chi2}
\textcolor{black}{\chi(F(m)) = \frac{1}{6}({c_{1}}^{3}-3c_{1}c_{2} + 3c_{3}) + \frac{1}{2}m({c_{1}}^{2}-2c_{2})+ \frac{m^{2}c_{1}}{2}+ \frac{m^{3}}{2} + ({c_{1}}^{2}-2c_{2}) + 2mc_{1} + 3m^{2} + \frac{11}{6}c_1 + \frac{11}{2}m + 3.}
\end{equation}

Comparing the coefficients of \eqref{chi1} and \eqref{chi2} on degree $0$ and $1$, we have $d= c_{2}$ and $c_{3}-4c_{2}+c_{1}c_{2}=2g-2$.

\end{proof}

Similiar as in the rank $2$ case we have the following result.

\begin{corollary}
Let $F$ be a rank $3$ vector bundle on $\mathbb{P}^3$, with Chern classes $c_1,c_2,c_3$. Then $c_3\equiv c_1c_2 (mod 2)$.
\end{corollary}

\textcolor{black}{As an application of Proposition \ref{chernclass} we can provide a converse for Theorem \ref{ratinst}.}

\begin{proposition}\label{instrat}Consider a rank $3$ instanton bundle $F$ on $\mathbb{P}^3$ of charge $n$, without global sections and suppose that $ h^{0}(F(1)) \geq 2$. Then we can write $F(1) \in \Ext^1(I_Y(3), \mathcal{O}_{\mathbb{P}^{3}}^{\oplus 2})$, where $Y$ is a rational curve in $\mathbb{P}^3$ of degree $n+3$.
\end{proposition}
\begin{proof}
Indeed, as $h^0(F(1)) \geq 2$ we get two linearly independent sections of $F(1)$ which determine an exact sequence

$$0 \rightarrow \mathcal{O}_{\mathbb{P}^{3}}^{\oplus 2} \rightarrow F(1)\rightarrow I_{Y}(3) \rightarrow 0$$

\noindent where $Y$ is a locally complete intersection curve in $\mathbb{P}^3$.
The Chern classes of $F(1)$ are $(3,n+3,n+1)$. Then using Theorem \ref{chernclass}, we have:

$$d=n+3$$
$$n+1-4(n+3)+3(n+3)=2g-2$$

\noindent Hence $g=0$ and $d=n+3$ and we conclude that $Y$ is a rational curve of degree $n+3$.
\end{proof}

The next proposition and its corollary will be useful to compute the dimension of irreducible components of the moduli spaces of rank $3$ vector bundles on $\mathbb{P}^3$ with given Chern classes.

\begin{proposition}\label{euler} Let $F$ be a rank $3$ vector bundle on $\mathbb{P}^3$ with Chern classes $c_{1}, c_{2}, c_{3}$. Then
$$\chi(F \otimes F^{*}) = 4{c_{1}}^{2}-12c_{2}+9.$$

\end{proposition}

\begin{proof} Let $F$ be a rank $3$ vector bundle and consider $H = c_{1}({\mathcal{O}}_{\mathbb{P}^{3}}(1))$, the class of a hyperplane in the Chow ring of $\mathbb{P}^{3}$, then 

$$ch(F) = 3 + c_{1}H + \frac{1}{2}({c_{1}}^{2}-2c_{2})H^{2}+\frac{1}{6}({c_{1}}^{3}-3c_{1}c_{2}+3c_{3})H^{3}$$
and

$$ch(F^{*}) = 3 - c_{1}H + \frac{1}{2}({c_{1}}^{2}-2c_{2})H^{2}+\frac{1}{6}(-{c_{1}}^{3}+3c_{1}c_{2}-3c_{3})H^{3}.$$

Hence

$$ch(F \otimes F^{*}) = ch(F)ch(F^{*})= 9 + 2({c_{1}}^{2}-3c_{2})H^{2}$$

\noindent and using Hirzenbruch-Riemann-Roch theorem we have:

$$\chi(F \otimes F^{*}) =  4{c_{1}}^{2}-12c_{2}+9.$$
\end{proof}

\begin{corollary} \label{ext} Let $F$ be a stable rank $3$ vector bundle on $\mathbb{P}^{3}$ with Chern classes $c_{1}, c_{2}, c_{3}$. Then:

$$\dim \Ext^{1}(F,F)-\dim \Ext^{2}(F,F) = -4{c_{1}}^{2}+ 12c_{2}-8.$$
\end{corollary}

\begin{proof} Since $F$ is stable, $ \dim \Hom(F,F) = 1$ and  $F\otimes F^{*}$ is semistable, thus

$$\begin{array}{rcl}
  \Ext^{3}(F,F)   &\simeq  &\Hom(F,F(-4))^{*}\\
     & \simeq& H^0(F\otimes F^{\ast}(-4))\\
     &\simeq &0.
\end{array}$$

{It follows from Proposition \ref{euler} that}

$$\begin{array}{rcl}
  4{c_{1}}^{2}-12c_{2}+9   &=  &\chi(F\otimes F^{*})\\
       & =& 1 -\dim \Ext^{1}(F,F)+\dim \Ext^{2}(F,F).
\end{array}$$

\end{proof}

\section{Rank 3 instanton bundles and rank 2 reflexive sheaves}\label{correspondence}

\subsection{Rank $3$ instanton bundles of charge $n$} Here we find smooth points in the moduli space of rank $3$ vector bundles of Chern classes $(c_1,c_2,c_3)=(0,n,0)$ by analysing the instanton bundles of charge $n$. 

Let $Y\subset \mathbb{P}^3$ be a nondegenerate rational smooth curve of degree $n+3$ and $F$ be its corresponding rank $3$ instanton bundle of charge $n$, by Theorem \ref{ratinst}.
{\color{black}Then we have the following lemma.}

\begin{lemma}\label{lem01}
Let $Y\subset \mathbb{P}^3$ be a \textcolor{black}{nondegenerate smooth rational curve} of degree $n+3$ and $F$ be its corresponding rank $3$ \textcolor{black}{instanton} bundle of charge $n$. If $\Ext^2(I_Y,I_Y)=0$, then
$\Ext^2(F,F)=0$.
\end{lemma}

\begin{proof}
{\color{black}Consider the exact sequence 
\begin{equation}\label{eq1}
\xymatrix{
0 \ar[r] & \mathcal{O}_{\mathbb{P}^3}^{\oplus 2} \ar[r] & F(1) \ar[r] & I_Y(3) \ar[r] & 0,
}
\end{equation} given by Theorem \ref{ratinst} and twist it by} $\mathcal{O}_{\mathbb{P}^3}(-1)$. Applying the functor $\Hom(-,F)$, we get the epimorphism 
\begin{equation}\label{eq6}
\xymatrix{
\cdots \ar[r] & \Ext^2(I_Y(2),F) \ar[r] & \Ext^2(F,F) \ar[r] & 0 
}
\end{equation}
since $F$ is a linear sheaf.

Twisting the exact sequence 
\begin{equation}\label{eq2}
\xymatrix{
0 \ar[r] & I_Y \ar[r] & \mathcal{O}_{\mathbb{P}^3} \ar[r] & \mathcal{O}_{Y} \ar[r] & 0 
}
\end{equation}
by $\mathcal{O}_{\mathbb{P}^3}(2)$ and applying the functor $\Hom(-,F)$, we get the isomorphism 
\begin{equation}\label{eq7}
\Ext^2(I_Y(2),F)\simeq\Ext^3(\mathcal{O}_{Y},F(-2)),
\end{equation}
since $H^2(F(-2))\simeq H^3(F(-2))=0.$

We observe that, by Serre Duality,
\begin{equation}\label{eq8}
\Ext^3(\mathcal{O}_{Y},F(-2))\simeq \Hom(F(1),\mathcal{O}_{Y}(-1))^{\ast} 
\end{equation}
and
\begin{equation}\label{eq9}
\Ext^3(\mathcal{O}_{Y},I_Y)\simeq \Hom(I_Y(3),\mathcal{O}_{Y}(-1))^{\ast}.
\end{equation}

Applying the functor $\Hom(-,\mathcal{O}_{Y}(-1))$ on the exact sequence $\eqref{eq1}$, we get the isomorphism

\begin{equation}\label{eq10}
\Hom(F(1),\mathcal{O}_{Y}(-1))\simeq \Hom(I_Y(3),\mathcal{O}_{Y}(-1)),
\end{equation}
since $H^0(\mathcal{O}_{Y}(-1))=0.$

Referring to $\eqref{eq6}$, $\eqref{eq7}$, $\eqref{eq8}$, $\eqref{eq9}$
and $\eqref{eq10}$
we conclude that if $\Ext^3(\mathcal{O}_{Y},I_Y)=0$, then 
$\Ext^2(I_Y(2),F)=0$ and, therefore, $\Ext^2(F,F)=0$.

Finally, applying the functor  $\Hom(-,I_Y)$ on the
exact sequence $\eqref{eq2}$, we get the isomorphism
\begin{equation}\label{eq11}
\Ext^2(I_Y,I_Y)\simeq \Ext^3(\mathcal{O}_{Y},I_Y),
\end{equation}
since $H^2(I_Y)=H^3(I_Y)=0.$

Therefore, if $\Ext^2(I_Y,I_Y)=0$, then $\Ext^2(F,F)=0$.
\end{proof}

\textcolor{black}{Recall that a rank $2$ instanton bundle on $\mathbb{P}^3$ $F$ is said to be a 't Hooft instanton bundle if it can be obtained as an extension of the form 
$$\xymatrix{0\ar[r] & \mathcal{O}_{\mathbb{P}^3}(-1) \ar[r] & F \ar[r] & I_{\Lambda}(1)\ar[r] & 0}$$
where $\Lambda$ is disjoint union of $n+1$ lines (see \cite[Proposition 7.2]{JMT2018}). In particular, one can see that $H^0(F(1))\geq \rk(F)-1=1$ which motivates the following definition.}

\begin{definition}\label{hooft}
\textcolor{black}{A {\bf generalized 't Hooft instanton bundle} is a rank $r$ instanton bundle $F$ on $\mathbb{P}^3$ such that $h^0(F(1))\geq r-1$.}
\end{definition}

\textcolor{black}{Similar to rank 2 case (see \cite{C1988} and \cite{NT1994}), next we prove that the stable generalized 't Hooft bundles correspond to smooth points in the moduli space of instanton bundles.}

\begin{theorem}\label{smooth}
The stable rank $3$ generalized 't Hooft instanton bundle of charge $n$ is a smooth point in its irreducible component in the moduli space of rank $3$ vector bundles of Chern classes $(c_1,c_2,c_3)=(0,n,0)$.
\end{theorem}
\begin{proof}
\textcolor{black}{Let $F$ be a stable rank $3$ generalized 't Hooft instanton bundle of charge $n$, by Proposition \ref{instrat} $F$ corresponds to a rational curve $Y\subset \mathbb{P}^3$ of degree $n+3$. Since smoothness and irreducibility are open conditions in its Hilbert scheme of rational curves of fixed degree, we have that the generic choice of sections of $H^0(F(1))$ will be associated with a nondegenerate smooth irreducible curve $Y$, which is smooth in that Hilbert scheme by \cite[Theorem 3.1.4 and Lemma 3.1.5]{L2000} with $H^1(N_{Y/\mathbb{P}^3}) = 0$, hence by \cite[Lemma B.5.6]{KPS2018}, we have that $Ext^2(I_Y,I_Y)=0$, from what follows that $\Ext^2(F,F)=0$. Therefore, $F$ is a smooth point in its irreducible component of the moduli space.}
\end{proof}

\subsection{Rank 3 instanton bundles of charge $2$ and rank 2 reflexive sheaves} The main goal of this section is to relate stable rank $3$ instanton bundles on ${\mathbb P}^3$ of charge $2$ with stable rank $2$ reflexive linear sheaves with Chern classes $(c_{1},c_{2},c_{3})=(1,3,3)$. Then we use this result to prove that the family of rank $3$ stable instanton bundles of charge $2$ has dimension 16. 

\begin{theorem} \label{equivalence} Let $Y \subset \mathbb{P}^{3}$ be \textcolor{black}{an irreducible} curve that satisfies the hypothesis of the \textcolor{black}{Hartshorne-Serre Correspondence}. Given a rank $3$ coherent sheaf $F$ on $\mathbb{P}^3$, then the following conditions are equivalent:
\begin{itemize}
    \item[1.] $F$ is a vector bundle which corresponds to $Y$ via the \textcolor{black}{Hartshorne-Serre Correspondence};
    \item[2.] $F \in \Ext^1(E,\mathcal{O}_{\mathbb{P}^3})$, where $E$ is a rank 2 reflexive sheaf which corresponds to $Y$ via the \textcolor{black}{Hartshorne-Serre Correspondence}.
    \end{itemize}
\end{theorem}
\begin{proof}
Let $F$ be a rank 3 vector bundle which corresponds to $Y$ via the \textcolor{black}{Hartshorne-Serre Correspondence}, take $\alpha, \beta \in H^0(F)$ and consider the following commutative diagram:

$$\xymatrix{0 \ar[r] & {\mathcal{O}_{\mathbb{P}^3}}\ar[d]^{i}\ar[r]^{\alpha} &F\ar@{=}[d]\\0 \ar[r] & {\mathcal{O}_{\mathbb{P}^3}^{\oplus 2}}\ar[r]_{(\alpha, \beta)} &
F\ar[r] & I_{Y}(c)\ar[r]& 0,}$$

\noindent where $i:\mathcal{O}_{\mathbb{P}^3}\rightarrow \mathcal{O}_{\mathbb{P}^3}^{\oplus 2}$ is the inclusion on the first summand \textcolor{black}{and $c:= c_1(F)$}. By the \textcolor{black}{snake lemma} we get the exact sequence

\begin{equation} \label{seqE}
\xymatrix{0\ar[r] &{\mathcal{O}_{\mathbb{P}^3}} \ar[r]^{\alpha} & E \ar[r] & I_{Y}(c) \ar[r] & 0, }
\end{equation}

\noindent where $E=\coker(\alpha)$ \textcolor{black}{and that $F \in \Ext^1(E,\mathcal{O}_{\mathbb{P}^3})$}. One can prove that $E$ is a reflexive sheaf. \textcolor{black}{Applying $\mathcal{H}om( -, \mathcal{O}_{\mathbb{P}^3})$ in sequence (13) we have
$$\xymatrix{0\ar[r] & \mathcal{H}om(I_Y(c), \mathcal{O}_{\mathbb{P}^3})\ar[r] & E^{\ast} \ar[r] &\mathcal{O}_{\mathbb{P}^3} \ar[r]^-{s} & \mathcal{E}xt^1(I_Y(c), \mathcal{O}_{\mathbb{P}^3})\ar[r] & \mathcal{E}xt^1(E, \mathcal{O}_{\mathbb{P}^3})\ar[r] &0.}$$
Note that $s\in H^0(\omega_Y(4-c))\simeq Ext^1(I_Y(c), \mathcal{O}_{\mathbb{P}^3})$. Since $Y$ is irreducible, then either $s=0$, therefore the extension is trivial and consequently $F$ is not a vector bundle, a contradiction; or $s$ vanishes in a finite number of points, hence $\dim \mathcal{E}xt^1(E, \mathcal{O}_{\mathbb{P}^3})=0$ and $E$ is reflexive.}

On the other hand, let $E$ be a rank 2 reflexive sheaf which corresponds to $Y$ and let $F \in \Ext^1(E,\mathcal{O}_{\mathbb{P}^3})$. Then we have the following commutative diagram

$$\xymatrix{& &0 \ar[d] &0\ar[d]\\ & &0\ar[d] \ar[r] &{\mathcal{O}_{\mathbb{P}^3}}\ar[d]\\ 0 \ar[r] & {\mathcal{O}_{\mathbb{P}^3}} \ar[r] & F \ar@{=}[d]\ar[r]^{p} & E \ar[d]^{s}\ar[r] & 0\\ & & F \ar[d] \ar[r]^-{s\circ p} & I_{Y}(c)\ar[d]\ar[r]& 0\\& &  0 & 0 }$$

By the \textcolor{black}{snake lemma} we get the short exact sequence

$$0 \rightarrow {\mathcal{O}_{\mathbb{P}^3}} \rightarrow \ker{(s\circ p)} \rightarrow {\mathcal{O}_{\mathbb{P}^3}} \rightarrow  0,$$ 

\noindent thus we have $\ker{(s\circ p)} \simeq  \mathcal{O}_{\mathbb{P}^3}^{\oplus 2}$ and the short exact sequence

\begin{equation}\label{seqF}
0 \rightarrow \mathcal{O}_{\mathbb{P}^3}^{\oplus 2} \rightarrow F \rightarrow I_{Y}(c) \rightarrow 0.    
\end{equation}

Finally, we need to prove that $F$ is locally free. First note that $F$ is a reflexive sheaf, indeed, dualizing the sequence 
$$0 \rightarrow {\mathcal{O}_{\mathbb{P}^3}} \rightarrow F \rightarrow E \rightarrow 0,$$ 
\noindent we get $\mathcal{E}xt^2(F,\mathcal{O}_{\mathbb{P}^3})=\mathcal{E}xt^3(F,\mathcal{O}_{\mathbb{P}^3})=0$ and the surjective map $\xymatrix{\mathcal{E}xt^1(E,\mathcal{O}_{\mathbb{P}^3})\ar[r]^{\delta}& \mathcal{E}xt^1(F,\mathcal{O}_{\mathbb{P}^3})}$. Since $E$ is reflexive it follows that $\mathcal{E}xt^1(E,\mathcal{O}_{\mathbb{P}^3})$ and thus $\mathcal{E}xt^1(F,\mathcal{O}_{\mathbb{P}^3})$ is supported in a finite number of points. Therefore, $F$ is a reflexive sheaf.
Now dualizing \eqref{seqF} we get
$$\xymatrix{\mathcal{O}_{\mathbb{P}^3}^{\oplus 2}\ar[r]^-{\lambda}&\mathcal{E}xt^1(I_Y(c),\mathcal{O}_{\mathbb{P}^3})\ar[r]& \mathcal{E}xt^1(F,\mathcal{O}_{\mathbb{P}^3})\ar[r]&0}.$$

Since $\mathcal{E}xt^1(I_Y(c),\mathcal{O}_{\mathbb{P}^3})\simeq \omega_Y(4-c)$, and by hypothesis, $\omega_Y(4-c)$ is globally generated by 2 sections, it is straightforward to check that $\lambda$ is either surjective (and thus $\mathcal{E}xt^1(F,\mathcal{O}_{\mathbb{P}^3}) = 0$) or $\lambda = 0$, and thus $F^{\ast} = \mathcal{O}_{\mathbb{P}^3}(-c)\oplus\mathcal{O}_{\mathbb{P}^3}^{\oplus 2}$, that is, $F = F^{\ast \ast  } = \mathcal{O}_{\mathbb{P}^3}(c)\oplus\mathcal{O}_{\mathbb{P}^3}^{\oplus 2}$. In both cases, $F$ is a vector bundle as we wanted. 

\end{proof}

\begin{theorem}\label{correspondence1}There exists a one to one correspondence between:
\begin{itemize}
\item[1.] Pairs $(F, \sigma)$, where $F$ is a charge $2$ rank $3$ stable instanton bundle on $\mathbb{P}^3$ and $\sigma$ is a global section of $F(1)$. 

\item[2.] Pairs $(E,\xi)$ where \textcolor{black}{$E$ is rank $2$ stable reflexive sheaf} on $\mathbb{P}^3$ with Chern classes $(-1,3,3)$ and $\xi \in \Ext^{1}(E(2),\mathcal{O}_{\mathbb{P}^{3}})$.
\end{itemize}
\end{theorem}
\begin{proof}

Given $(F, \sigma)$, where $F$ is a charge $2$ rank $3$ stable instanton bundle on $\mathbb{P}^3$ and $\sigma \in H^0(F(1))$. Since $6=\chi(F(1))=h^0((F(1))-h^1((F(1))$ it follows that $h^0((F(1))\geq 6$, and by Proposition \ref{instrat} $F(1)$ corresponds to a rational quintic $Y\subset \mathbb{P}^3$ via the \textcolor{black}{Hartshorne-Serre Correspondence}.\\

By Theorem \ref{equivalence}, $F(1)\in \Ext^1(N, \mathcal{O}_{\mathbb{P}^3})$, where $N$ is a rank 2 reflexive sheaf \textcolor{black}{with Chern classes $(3,5,3)$} which also corresponds to $Y$ via \textcolor{black}{Hartshorne-Serre Correspondence}. By \cite[Proposition 4.2]{H1980b} $N$ is stable. Now let $E:=N_{norm}=N(-2)$, thus we have that the Chern classes of $E$ are $(-1,3,3)$ and $F(1) \in \Ext^{1}(E(2),\mathcal{O}_{\mathbb{P}^{3}})$.

On the other hand, let $(E,\xi)$, where $E$ is normalized rank $2$ stable reflexive sheaf on $\mathbb{P}^3$ with Chern classes $(-1,3,3)$ and $\xi \in \Ext^{1}(E(2),\mathcal{O}_{\mathbb{P}^{3}})$. By \cite[Theorem 3.13]{C1984}, $E(2)$ corresponds to a rational quintic $Y\subset \mathbb{P}^3$ via \textcolor{black}{Hartshorne-Serre Correspondence}. Thus by Theorem \ref{equivalence} $\xi$ defines a rank 3 vector bundle $G$ on $\mathbb{P}^3$ with Chern classes $(1,3,3)$. Let $F:=G_{norm}=G(-1)$, by Theorem \ref{ratinst} $F$ is a instanton bundle, and $\sigma$ is defined by the map $\mathcal{O}_{\mathbb{P}^3}\to F(1)$ in the extension

\begin{equation}\label{eq90}
    \xymatrix{0\ar[r]& \mathcal{O}_{\mathbb{P}^3}\ar[r]^-{\sigma}& F(1)\ar[r]^{\beta} & E(2)\ar[r]& 0}
\end{equation}

Next we prove that $F$ is slope stable, which would imply that $F$ is stable by \cite[Lemma 1.2.12]{OSS1980}. The sequence \eqref{eq90} gives us $H^{0}(F) = 0$, that is, there is no destabilizing coherent subsheaf of rank $1$. For the rank $2$ case, assume $\phi : F \rightarrow L$ a surjection where $L$ is a \textcolor{black}{rank $1$ torsion free sheaf} such that ${c_{1}}(L) \leq 0$. If ${\phi|}_{\mathcal{O}_{\mathbb{P}^{3}}(-1)} = 0$ then there exists a map \textcolor{black}{$u: E(1) \rightarrow L$ such that $\phi = u \circ \beta$. Hence $u: E(1) \rightarrow L$ is a surjective map and ${c_{1}}(L) \leq 0 < \frac{1}{2}$ which contradicts the stability of $E$}. Thus the sheaf $\mathcal{O}_{\mathbb{P}^{3}}(-1)$ is injected onto $L$ and we have $L \cong \mathcal{O}_{\mathbb{P}^{3}}(-1)$ so the sequence \eqref{eq90} splits which is a contradiction. This concludes the proof.
\end{proof}

In what follows we use this result to prove that every rank $3$ instanton bundle of charge 2 has natural cohomology.

\begin{definition}
We say that a vector bundle $F$ on $\mathbb{P}^3$ has \textit{natural cohomology} if for
any integer $t$ at most one cohomology group $H^i(F(t))$ is nontrivial for $0\leq i\leq 3$.
\end{definition}

\begin{proposition}\label{natural}
Every rank 3 stable instanton bundle on $\mathbb{P}^3$ with charge 2 has natural cohomology. 
\end{proposition}

\begin{proof}
Let $F$ be a rank 3 stable instanton bundle on $\mathbb{P}^3$ with charge 2.
By Theorem \eqref{correspondence1}, $F$ satisfy the exact sequence in display 
\eqref{eq90}. \textcolor{black}{On the other hand, by \cite[Theorem 3]{J2006} $F$ is a linear sheaf, and therefore cohomology of a linear monad of the form $$0 \rightarrow \mathcal{O}_{\mathbb{P}^{3}}(-1)^{\oplus 2} \rightarrow \mathcal{O}_{\mathbb{P}^{3}}^{\oplus 7}\rightarrow \mathcal{O}_{\mathbb{P}^{3}}(1)^{\oplus 2} \rightarrow 0.$$
    Tensoring this monad by $\mathcal{O}_{\mathbb{P}^3}(t)$, with $t\in \mathbb{Z}$. From the long exact sequence of cohomology in the display of the monad we have the following:}

If $t\geq 1$, by  \cite[Table 3.13.1]{C1984}, we have that $h^1(F(t))=h^1(E(t+1))=0$ and \textcolor{black}{$h^2(F(t))=h^3(F(t))=0$} since $F$ is a linear sheaf.

If $-2\leq t\leq 0$, we have $h^0(F(t))=0$ because $F$ is stable and $h^2(F(t))=h^3(F(t))=0$ since $F$ is a linear sheaf. 

If $t=-4,-3,$ then $h^0(F(t))=h^1(F(t))=h^3(F(-3))=0$ since $F$ is a linear sheaf
and, by Serre duality, $h^3(F(-4))=h^0(F^{\ast})=0$ because $F$ is stable 
(see \cite[Remark 1.2.6]{OSS1980}).

If $t\leq -5$, we have $h^0(F(t))=h^1(F(t))=0$ since $F$ is a linear sheaf.
By Serre duality, $h^2(F(-5))=h^1(F^{\ast}(1))=0$ since $F^{\ast}$ is also a rank
3 stable instanton bundle of charge 2.

The exact sequence in display \eqref{eq90} induces the exact sequence in cohomology
$$0 \to H^2(F(t)) \to H^2(E(t+1)) \to \cdots $$
\noindent Since $E^{\ast}\simeq E(1)$, we get $h^2(E(t+1))=h^1(E(-t-4))=0$ for
$t\leq -6$ (see \cite[Theorem 3.13]{C1984}). Thus,
$h^2(F(t))=0$ for all $t\leq -6.$

Therefore $F$ has natural cohomology.
 
\end{proof}

 \section{Moduli Space of Rank $3$ bundles with Chern classes $(c_1, c_2, c_3)=(0,2,0)$} \label{moduli}

Now we are in position to study $\mathcal{I}(2)$ the family of rank 3 stable instanton bundles with charge $2$ on $\mathbb{P}^3$. Our goal is to prove that \textcolor{black}{the closure of $\mathcal{I}(2)$ in $\mathcal{B}(0,2,0)$ the moduli space of rank 3 stable vector bundles on $\mathbb{P}^3$ is an irreducible component}. To do so, we will use the correspondence \textcolor{black}{given by Theorem \ref{correspondence1}.}

Let $\mathcal{R}$ be the moduli space of stable rank $2$ reflexives sheaves with Chern classes $(-1,3,3)$ and $\mathcal{U}\to \mathbb{P}^3 \times \mathcal{R}$ its corresponding universal bundle. By \cite[Theorem 3.13]{C1984} $\mathcal{R}$ is an irreducible rational variety of dimension $19$.

Let $p: \mathbb{P}^3\times \mathcal{R} \to \mathcal{R}$ be the projection onto the second coordinate. For every $[E] \in \mathcal{R}$, by the base change theorem for relative Ext-sheaves \cite[Theorem 1.4]{L1983} one has
$$\mathcal{E}xt_p^i(\mathcal{U}(2,0), \mathcal{O}_{\mathbb{P}^3 \times \mathcal{R}})\otimes_E \mathbb{K}([E])\simeq \Ext^i(E(2), \mathcal{O}_{\mathbb{P}^3})\simeq H^{3-i}(E(-2)),\textrm{ for }i=0,1,2,3,$$ hence by \cite[Table 3.13.1]{C1984} one has
$\mathcal{E}xt^i_p(\mathcal{U}(2,0), \mathcal{O}_{\mathbb{P}^3 \times \mathcal{R}})=0,$ for $i=0,2,3.$
Now, by the local-to-global spectral sequence for Ext-sheaves one has $\Ext^1(\mathcal{U}(2,0), \mathcal{O}_{\mathbb{P}^3 \times \mathcal{R}})\simeq H^0(\mathcal{E}xt^1(\mathcal{U}(2,0), \mathcal{O}_{\mathbb{P}^3 \times \mathcal{R}}))$, therefore $\dim \Ext^1(\mathcal{U}(2,0), \mathcal{O}_{\mathbb{P}^3 \times \mathcal{R}})=3$ and $\mathcal{X}=\mathcal{R}\times \Ext^1(\mathcal{U}(2,0), \mathcal{O}_{\mathbb{P}^3 \times \mathcal{R}})$ is an irreducible variety of dimension 22 that parametrizes the pairs $(E,\xi)$, with $\xi \in \Ext^{1}(E(2),{\mathcal O}_{\mathbb{P}^{3}})$.
Consider  $\varphi: \mathcal{X} \to \mathcal{B}(0,2,0)$ the map given by the composition of the correspondence of the Theorem \ref{correspondence1}, and the projection onto the first coordinate. Clearly, $\varphi$ factorizes trough $\tilde{\varphi}:\mathcal{X} \to \mathcal{I}(2)$. Additionally, for every \textcolor{black}{$F \in \mathcal{I}(2)$, one has that $\dim \tilde{\varphi}^{-1}(F) = h^0(F(1))$. But $6=\chi(F(1))=h^0((F(1))-h^1((F(1))$ and since $F$ has natural cohomology by Proposition \ref{natural} it follows that $\dim \tilde{\varphi}^{-1}(F) = h^0(F(1))=6$}. Thus, $\tilde{\varphi}$, by construction is an surjective open morphism of relative dimension $6$, which implies that $\tilde{\varphi}(\mathcal{X}) = \mathcal{I}(2)$, is an irreducible variety of dimension $16$. Thus we have the following result.

\begin{proposition}\label{cordim}The family of stable rank $3$ instanton bundles of charge $2$  is irreducible and has dimension $16$.  
\end{proposition}

Now, our next goal is to prove that $\mathcal{I}(2)$ is in fact the whole moduli space $\mathcal{B}(0,2,0)$. First we have the following.

\begin{proposition} \label{I(2)}
    Every rank $3$ stable vector bundle on $\mathbb{P}^3$ of Chern classes $(0,2,0)$ is an instanton bundle.
\end{proposition}
\begin{proof}
Let $F$ be a rank $3$ stable vector bundle on $\mathbb{P}^3$ of Chern classes $(0,2,0)$ and $k_F=(k_1,k_2)$ its \textcolor{black}{spectrum}. By the properties of \textcolor{black}{spectrum} \cite{OS1984} we have two possibilities $k_F=(0,0)$ or $k_F=(-1,1)$. If the first occurs by Proposition \ref{instanton1} we have the result. Let us prove that $k_F=(-1,1)$ it is not admissible. In fact, by Theorem \ref{spectro} we have
\begin{equation}\label{cohoF}
    h^1(F(-2))=h^2(F(-2))=1, ~ h^2(F(1))=0.
\end{equation}
By Hirzenbruch-Riemann-Roch formula we have $\chi(F(1))=6$, thus $h^0(F(1))\geq 6$ and by Proposition \ref{chernclass}, there exists a locally complete intersection rational quintic $Y\in \mathbb{P}^3$ that corresponds to $F(1)$ via the Hartshorne-Serre correspondence. \textcolor{black}{Therefore we have an exact sequence
\begin{equation}
0\rightarrow {\mathcal O}_{\mathbb{P}^{3}}^{\oplus 2} \rightarrow F(1) \rightarrow I_{Y}(3) \rightarrow 0.
\end{equation}
Note that $h^0(I_{Y}(1))=h^0(F(-1))=0$, where the second equality follows from the stability of $F$. Now, analogously to the proof of Theorem \ref{ratinst} we can conclude that $F$ is an instanton bundle of charge $2$, since all the needed vanishing of the cohomologies follow from degree and the genus of $Y$. Therefore} $h^1(F(-1)) = 2$ which contradicts the dimensions given by \eqref{cohoF}.
\end{proof}

Finally we have our main result.

\begin{theorem}\label{mainresult}
$\mathcal{B}(0,2,0)$ the moduli space of rank $3$ stable vector bundles on $\mathbb{P}^3$ with Chern classes $(0,2,0)$ is irreducible of dimension $16$ and its generic point corresponds to a \textcolor{black}{generalized }t` Hooft instanton bundle.
\end{theorem}
\begin{proof}
Proposition \ref{I(2)} and Proposition \ref{cordim} imply that $\mathcal{B}(0,2,0)=\mathcal{I}(2)$ is irreducible of dimension $16$. Now, let $F\in \mathcal{I}(2)$ be a generic generalized 't Hooft \textcolor{black}{instanton} bundle, then by Theorem \ref{smooth} $\Ext^2(F,F)=0$ and by Corollary \ref{ext} we have that $\dim \Ext^1(F,F)=16$ which coincides with dimension of $\mathcal{I}(2)$ computed in Theorem \ref{cordim}. This implies that the generic generalized 't Hooft instanton bundle is a smooth point of $\mathcal{B}(0,2,0)$. 

\end{proof}

\bibliographystyle{amsalpha}

\frenchspacing

\end{document}